\documentclass[12pt,reqno]{amsart}

\usepackage{amsthm,amsfonts,amssymb,amsmath,epsf}

  \newtheorem{theorem}{Theorem}
  \newtheorem{lemma}[theorem]{Lemma}
  \newtheorem{corollary}[theorem]{Corollary}

\begin{document}

\title[ ]{Upper bounds for the number of primitive ray class characters with conductor below a given bound}

\author{Joshua Zelinsky}
  \date{}
  \email{jzelinsk@bu.edu}
  \address{Department of Mathematics and Statistics, Boston University, 111 Cummington Mall, Boston, MA 02215}

  \thanks{David Rohrlich looked at multiple drafts and offered numerous helpful comments.}
  \maketitle
{\bf Abstract.}
We present upper bounds on certain sums  which are related to Artin's primitive root conjecture and are also used in counting ray class characters.

\section{Introduction}

Fix a positive integer $a$, and  write $\mathrm{ord}_n(a)$  for the order of $a$ in the multiplicative group of invertible residue classes modulo $n$ when $(a,n)=1$. Let $$G(x)= \sum_{n \leq x, (n,a)=1} \frac{\phi(n)}{\mathrm{ord}_n(a)}.$$ We show that for any $\alpha <3$, we have $G(x) = O(x^2/ \log^\alpha x)$.  The motivation for investigating these sums stems from two distinct problems: the Artin primitive root conjecture and the problem of counting Artin representations.

Artin conjectured that for any given a rational integer $a$, $a \neq 1$ and $a$ not a perfect square,  the set rational primes $p$ such that $a$ is a primitive root mod $p$ has positive density and he conjectured a formula for that density. Major work on Artin's conjecture is due to Hooley \cite{Hooley}, Murty and Gupta\cite{Murty}, and Heath-Brown\cite{Heath-Brown}.  Since then, work has been done generalizing the Artin conjecture in a variety of directions.
Our result concerning $G(x)$ represents one such direction. Instead of confining ourselves to prime moduli, we consider also composite moduli and instead of considering moduli for which $a$ is of maximal order, we consider the average order of the index of the cyclic group generated by the residue class of $a$.

Our second motivating problem, the problem of counting Artin representations, is actually connected not to $G(x)$ itself but rather to a closely related sum over number fields. Let $K$ be a number field with ring of integers $O_K$. Let $U_K$ be the set of units of $O_K$, and for any ideal $I$ let $U_K(I)$ be the subgroup of $U_K$ formed by elements which are 1 (mod I). Let $U_K(I)^+$ be the subgroup of $U_K(I)$ formed by elements which are positive in all real embeddings of $K$.  Let $$P_K(x)= \sum_{\textbf{N}I \leq x} \frac{\phi(I)}{[U_K:U_K(I)]}$$ with the sum over non-zero ideals of $O_K$. Assume that the unit group of $O_K$ has positive rank. Then for any $\alpha <3$, we have $P_K(x) = O(x^2/ \log^\alpha x)$. $P_K$ represents an analog of $G$ for number fields.

Let us now explain the connection between these sums and the problem of estimating the number of Artin representations of fixed dimension, fixed base field, and  conductor of bounded norm. We focus on the case of dimension 1, so one is essentially counting primitive ray class characters. Let $K$ be a number field and let $\delta_{K,1}(x)$ count the number of 1-dimensional Artin representations up to isomorphism with norm of the conductor at most $x$.  Rohrlich\cite{Rohrlich} has noted the elementary estimate $\delta_{K,1}(x) = O(x^2)$, and if $K$  is the rationals or a quadratic imaginary field, then this is the correct order of growth and one can in fact produce an asymptotic formula, with the constant depending on the field. The method of proof is to note that one has
$$\delta_{K,1}(x) \leq \sum_{\textbf{N} I \leq x}{h_K}^{nar}(I)$$ where ${h_K}^{\mathrm{nar}}(I)$ is the order of the narrow ray class group of $K$ with modulus $I$. In fact, one has $$\delta_{K,1}(x) = \sum_{\textbf{N} I \leq x}\sum_{Q|I}\mu(\frac{I}{Q}){h_K}^{nar}(Q)$$ where $\mu$ is the generalization of the mobius function to ideals.

One has \cite{Lang2} $${h_K}^{\mathrm{nar}}(I)= \frac{2^{r_1}h_K\phi(I)}{[U_K:{U_K}^+(I)]}.$$  Note that ${[U_K:{U_K}^+(I)]}$ and ${[U_K:{U_K}(I)]}$ differ by at most a power of 2, the maximum exponent of which is bounded in terms of the rank of the unit group of the field. Thus, estimating $P_K(x)$ gives us  information about the number of Artin representations, since the other parts in the formula for the order of the  narrow class group are all dependent only on $K$.  Prior to this work, lower bounds for both $P(x)$ and $G(x)$ have been obtained by  Ambrose\cite{Ambrose}, who also conjectured that for both functions, the correct growth order would be $x^{2+o(1)}$. The estimates for $G(x)$ and $P_K(x)$ are similar, as one can think of $G(x)$ as the analog to $P_K(x)$ in the $S$-integers with $S$ equal to the set of prime divisors of $a$, but the proofs are presented separately; it is likely that a framework can be constructed which subsumes both results into a single proof, but as of yet, attempts at such an approach lead to technical difficulties.

\section{Over the rational integers}

\begin{lemma} If $2 \leq y \leq \frac{x}{2}$, then $\frac{x^x}{(x-y)^{x-y}} < e^{2y}x^y$.
\label{lemma1}
\end{lemma}
\begin{proof}
Consider
$$\log \frac{x^x}{(x-y)^{x-y}} = x\log x -  (x-y)\log(x-y)=x\left(\log x - (1-\frac{y}{x})\log(x-y) \right).$$ We have $$x\left(\log x - (1-\frac{y}{x})\log(x-y) \right) < x\left(\log\frac{x}{x-y} + \frac{y \log x}{x}\right).$$
Set $u=\frac{x}{x-y}-1=\frac{y}{x-y}$. Since $y \leq x/2$ and $\log(1+u) \leq u$ we have $\log(\frac{x}{x-y}) \leq \frac{2y}{x}$.
Thus we have $\log \frac{x^x}{(x-y)^{x-y}} \leq x\left(\frac{2y}{x} + \frac{y \log x}{x} \right)=2 y + y\log x$, so exponentiating
now gives the desired result.

\end{proof}

\begin{lemma} If $1 \leq j \leq \frac{k-2}{3}$, then ${k \choose j} \leq \frac{{k \choose j+1}}{2}$.
\label{lemma2}
\end{lemma}
\begin{proof} Note that ${k \choose j}/{k \choose j+1}= (j+1)/(k-j)$.
We have $\frac{j+1}{k-j} \leq \frac{1}{2}$ when $j \leq \frac{k-2}{3}$.
\end{proof}

\begin{lemma} If $2 \leq m \leq \frac{k-2}{3}$ then $\sum_{j=1}^m {k \choose j} \leq \left(\frac{e^2 k}{m}\right)^m$.
\label{HRlemma}
\end{lemma}
\begin{proof} We will first estimate ${k \choose m}$ and then use Lemma \ref{lemma2}  to bound $\sum_{j=1}^m {k \choose j}$.
We require the following version of Stirling's formula, valid for $n \geq 2$,:
$$\sqrt{2\pi n}\left(\frac{n}{e}\right)^n e^{\frac{1}{12n+1}} < n! < \sqrt{2\pi n}\left(\frac{n}{e}\right)^n e^{\frac{1}{12n}}$$ (See for example Lang's {\it Undergraduate Analysis\cite{Lang}.}) We have $$ \frac{k!}{m!(k-m)!} < \frac{\sqrt{2\pi k}\left(\frac{k}{e}\right)^k e^{\frac{1}{12k}} }{\sqrt{2\pi m}\left(\frac{m}{e}\right)^m \sqrt {2\pi (k-m)}\left(\frac{k-m}{e}\right)^{k-m}e^{\frac{1}{12m+1}} e^{\frac{1}{12(k-m)+1}} }.$$
Since $m < k$, we obtain:
$${k \choose m} < \frac{\sqrt{k}k^k}{\sqrt{m}m^m \sqrt{2\pi (k-m)}(k-m)^{k-m}}.$$ Applying Lemma \ref{lemma1} with $x=k$ and $y=m$ , we have $\frac{k^k}{(k-m)^{k-m}} < e^{2m}k^m$,
and thus
$${k \choose m} < \frac{e^{2m}k^m\sqrt{k}}{m^m\sqrt{m}\sqrt{2\pi (k-m)}}.$$ Since $m \leq \frac{k-2}{3}<\frac{k}{2}$, $k-m >\frac{k}{2}$ and so,
$${k \choose m} < \frac{e^{2m}k^m\sqrt{k}}{m^m\sqrt{m}\sqrt{\pi k }} = \frac{e^{2m}k^m}{\sqrt{m}\sqrt{\pi}m^m}.$$ By Lemma \ref{lemma2},
we have that $$\sum_{j=1}^m{k \choose j} \leq {k \choose m} + \frac{{k \choose m}}{2} + \frac{{k \choose m}}{4} \cdots = 2{k \choose m}< 2\frac{e^{2m}k^m}{\sqrt{m}\sqrt {\pi}m^m}$$
Since $m \geq 2$, $\frac{2}{\sqrt{\pi m}} < 1$ so we are done.
\end{proof}

\begin{lemma} For $\beta>0$, let $D_\beta(x)$ be the number of positive integers $n$ with $\omega(n) \geq (\log x)^\beta$. Then for any fixed $\beta$, we have
$D_\beta(x) = O(\frac{x (\log \log)^2}{(\log x)^{2\beta}})$, where the implied constant depends on $\beta$.
\label{Hardy1}
\end{lemma}
\begin{proof}
This follows from 22.11.6 (page 357) in Hardy and Wright\cite{Hardy}
\end{proof}

\begin{theorem} For any $\alpha <3$, we have $G(x) = O(\frac{x^2}{\log^\alpha x})$
\end{theorem}
\begin{proof} We will prove the equivalent result that  for any fixed $\alpha$ such that $1< \alpha< 3$, we have $G(x) = O(\frac{x^2 (\log \log x)^2}{(\log x)^\alpha})$.

Let $S(x)$ be the set of $n \leq x$ with  $(n,a)=1$. Define $S(x,p,q)$ to be the subset of $S(x)$ with $p \leq \mathrm{ord}_n(a) \leq q$. Set $t = (\log x)^\alpha$.  Note that if $x$ is sufficiently large we have that $t > 3\log_a x$.  Write $\ell(x)=3\log_a x$. Then we have:

\begin{equation}\label{Gbreakdown} G(x) \leq \mathrm{I}(x) + \mathrm{II}(x) + \mathrm{III}(x)  \end{equation}
where \begin{equation} \label{Idef1} \mathrm{I}(x)= \sum_{n \in S(x,1,\ell(x))} \frac{n}{\mathrm{ord}_n(a)}\end{equation} \begin{equation}\label{IIdef1} \mathrm{II}(x)= \sum_{n \in S(x,\ell(x),t)} \frac{n}{\mathrm{ord}_n(a)}\end{equation} and \begin{equation}\label{IIIdef1} \mathrm{III}(x)= \sum_{n \in S(x,t, \infty) } \frac{n}{\mathrm{ord}_n(a).}\end{equation}  We will need to estimate each of these sums separately.

The easiest two sums to estimate are $\mathrm{I}(x)$ and $\mathrm{III}(x)$. $\mathrm{I}(x)$ has at  most $\sum_{1 \leq r \leq \ell(x)} \tau(a^r-1)$ terms. We have
$$\sum_{\substack{1 \leq r \leq \ell(x) }} \tau(a^r-1) \leq \ell(x)\tau_m(x)$$ where $\tau_m(x)$ is the maximum of $\tau(a^r-1)$ with $r$ ranging
from $1$ to $\ell(x)$. Since $\tau(n)=O(n^{1/6})$, $\tau_m(x)=O(x^{1/2})$. Any term in $I(x)$ is bounded above by $x$, and so we conclude that
$I(x)= O(x^{1/2} \ell (x)x) = O(x^{3/2}\log x)$.

Next, we need to estimate $\mathrm{III}(x)$. In this case we use that it has at most $x$ terms and each term is at most $x/t$ and so $\mathrm{III}(x) \leq x^2/t$.

To estimate $\mathrm{II}(x)$ we need some preliminary remarks.

Given any $\epsilon >0$,  we have $\omega(n) < \frac{(1+\epsilon)\log n}{\log \log n}$ for all but finitely many $n$. For simplicity, we take $\epsilon=1$ and so for all but finitely many $n$ we have  $\omega(n) < \frac{2\log n}{\log \log n}$. Let $N$ be the set of $d$ such that
$(a,d)=1$ and $d|n$ for some $n$ satisfying $\omega(n) \geq \frac{2\log n}{\log \log n} $. Note that $N$ is finite.
Since $N$ is finite, $\sum_{n \in N, n \leq x} \frac{\phi(n)}{\mathrm{ord}_n(a)} = O(1)$.

Set $\beta = \frac{\alpha-1}{2}$. Set $H(x)$ to be the subset of elements of $S(x,\ell(x),t)$ and not in $N$ which have at most $(\log x)^\beta$ distinct prime factors,  and set \begin{equation}\label{II1def1} \mathrm{II}_1(x)= \sum_{n \in H(x)} \frac{n}{\mathrm{ord}_n(a)}.\end{equation} Similarly, let $J(x)$ be the elements with more than $(\log x)^\beta$ distinct prime factors and not in $N$
and set  \begin{equation}\label{II2def1} \mathrm{II}_2(x) =  \sum_{n \in J(x)} \frac{n}{\mathrm{ord}_n(a)}.\end{equation} Thus, $\mathrm{II}(x)= \mathrm{II}_1(x) + \mathrm{II}_2(x) +O(1)$.  So we need only bound $\mathrm{II}_1(x)$ and $\mathrm{II}_2(x)$.

To estimate $\mathrm{II}_1(x)$, we need to estimate $H(x)$. If $d \in H(x)$, then $d|a^r-1$ with $\ell(x) \leq r \leq t$. So we need to estimate how many
 divisors $a^r-1$ can have which are less than or equal to  $x$ and not in $N$.  Note that the number of prime factors of $a^r-1$ is bounded by  $$\omega(a^r-1) \leq \frac{2\log a^r}{\log \log a^r} \leq \frac{(3\log a)r}{\log r} \leq t.$$ The last inequality above is valid for sufficiently large values of $t$. Recall that $t$ grows with $x$.

 Thus,  each $a^r-1$ whose divisors contribute to $H(x)$ has at most $t$ distinct prime divisors.
However, each element in $H(x)$ has at $j$ distinct prime divisors, for some $j$ satisfying $1 \leq j \leq (\log x)^\beta$. So for any given $r$, there are at most at most $\sum_{j=1}^m {k \choose j}$ possible choices for the distinct prime divisors where $k$ is the largest integer less than or equal $t=(\log x)^\alpha$
and $m$ is the largest integer less than or equal to $(\log x)^\beta$. Any prime factor of such a divisor can
be raised to at most the $\log_2 x$ power, so the total number of divisors is bounded by  $(\sum_{j=1}^m {k \choose j})(\log_2 x)^{{(\log x)}^{\beta}}.$ There are at most $t$ possible values of $r$.  Thus,
$$|H(x)| \leq t\sum_{j=1}^m {k \choose j}(\log_2 x)^{(\log x)^\beta}.$$  Applying Lemma \ref{HRlemma} and using the values of $k$ and $m$ then gives $$ |H(x)|\leq t\left(\frac{e^2 t }{(\log x)^\beta }\right)^{(\log x)^\beta} (\log_2 x)^{(\log x)^\beta}.$$ Here we can replace the floor of $(\log x)^\beta$ since $\left(\frac{e^2 t }{(s}\right)^s$ is an increasing function in $s$ when $s$ is small compared to $t$.

 Thus we have $$\log |H(x)| \leq C(\log x)^\beta(\log \log x) + \log t + (\log \log_2 x)(\log x)^\beta  $$ for some constant $C$ and this is $O((\log x)^\beta \log \log x)$. Thus $|H(x)| = O(x^{\epsilon_0})$, and so $II_1(x) = O(x^{1+\epsilon_0})$ for any $\epsilon_0>0$.

To estimate  $\mathrm{II}_2(x)$ we apply Lemma \ref{Hardy1} so that the number of terms of $J$ is $O(\frac{x (\log \log x)^2}{(\log x)^{2\beta}})$ and thus conclude that $$\mathrm{II}_2(x) = O\left(\frac{x^2 (\log \log x)^2}{(\log x)^{2\beta +1}}\right)$$ since every
term in $\mathrm{II}_2(x)$ is at most $\frac{x}{\ell(x)}$.
\end{proof}

\section{Number fields}

Let $K$ be a number field and let $O_K$ be its ring of integers. Set $d= [K:Q]$. Define  $\omega(I)$ to be the number
of distinct prime ideal divisors of $I$, where $I$ is an ideal of $O_K$. Define  $j_K(x)= \sum_{\textbf{N}I \leq x} \omega(I)^2$.

\begin{lemma} $j_K(x)= O(x(\log \log x )^2)$.
\label{idealboundjk}
\end{lemma}
\begin{proof} We generalize the argument in Hardy and Wright (who treat the case $K=Q$). Note that $$\sum_{\substack{\textbf{N}I \leq x \\ \omega(I)=1}} \omega(I) = O(x).$$ But if $\omega(I) \geq 2$, then $$\omega(I)^2 - \omega(I) \geq \frac{\omega(I)^2}{2}.$$  Thus, we have

$$j_K(x) \leq 2\sum_{\textbf{N}I \leq x } \omega(I)(\omega(I)-1)  +O(x).$$ Thus, we need only show that
 $$\sum_{\textbf{N}I \leq x } \omega(I)(\omega(I)-1)= O(x(\log \log)^2) .$$ For a given ideal $I$, consider distinct prime ideals $P$ and $Q$ which divide $I$. There are $\omega(I)$ possible choices for $P$ and $\omega(I)-1$ possible choices for $Q$. Thus we have $$\sum_{\textbf{N}I \leq x } \omega(I)(\omega(I)-1) = \sum_{\textbf{N}I \leq x}\left(\sum_{PQ|I} 1\right) $$ where the sum is over distinct prime ideals. So we need only estimate $$\sum_{\textbf{N}I \leq x}\left(\sum_{PQ|I} 1\right).$$

Now, let $k$ be a constant such that the number of ideals of norm at most $x$ is bounded  by $kx$. Then,
$$\sum_{\textbf{N}I \leq x}\sum_{PQ|I} 1 \leq \sum_{\textbf{N}(PQ) \leq x} \frac{kx}{\textbf{N}(PQ)}$$ where the sum is over $P, Q$ prime. So if we can show that $$\sum_{\textbf{N}(pq) \leq x} \frac{x}{\textbf{N}(PQ)} = O(x (\log \log x)^2)$$ we will be done.

There are two contributions to this sum, terms where $\textbf{N} P$ and $\textbf{N} Q$ are both prime ideals of degree 1, versus terms where at least one prime is of degree greater than 1. Consider first the terms where at least one of $\textbf{N} P$ and $\textbf{N} Q$ is a power $\geq 2$ of a prime. Without loss of generality, consider the terms where $\textbf{N} P$ is a power $\geq 2$ of a prime. Since there are at most $d$ primes ideals of norm $p^a$ (where $p$ is a rational prime) and similarly at most $d$ primes of norm $q$ where $q$ is a rational prime, we then have: $$\sum_{\textbf{N}{PQ} \leq x, \textbf{N}P=p^a} \frac{x}{\textbf{N}(PQ)} \leq \sum_{\substack{a\geq 2 \\ p, q \leq x}} \frac{d^2 x }{p^a q}.$$ Here the sum on the left-hand of the inequality is over prime ideals $P$ of degree greater than 1, and  over prime ideals $Q$ of any degree.   We have $\sum_{a \geq 2}\frac{1}{p^a} \leq \frac{2}{p^2}$ and so
$$\sum_{\substack{a\geq 2 \\ p, q \leq x}} \frac{d^2 x }{p^a q} \leq \sum_{{q \leq x, p \leq x^{1/2}}} \frac{2d^2 x }{p^2 q} \leq \left(\sum \frac{2d^2}{p^2}\right) \sum_{q^i \leq x}{\frac{x}{q^i}} = O(x \log \log x.)$$

Thus we need to only examine $\sum_{\textbf{N} PQ \leq x} \frac{x}{\textbf{N} PQ}$ with the sum taken over ideals with prime norm.
This sum is bounded by $$\sum_{p,q \leq x}\frac{x}{pq} \leq  x \sum_{p,q \leq x}\frac{1}{pq} \leq x(\sum_{p \leq x} 1/p)^2 = O(x (\log \log x)^2).$$
\end{proof}

As an immediate corollary of the above we obtain:

\begin{lemma}
\label{omegacountbound2}
For $\beta>0$, let $D_\beta(x)$ be the number of ideals $I$ with $\omega(I) \geq (\log x)^\beta$, and $\textbf{N}I \leq x$. Then for any fixed $\beta$, we have
$D_\beta(x) = O(\frac{x (\log \log)^2}{(\log x)^{2\beta}})$, where the implied constant depends on $\beta$.
\end{lemma}

For a matrix $A$ with coefficients in $\mathbb{C}$, we will denote by $M(A)$  the maximum of the absolute value of the
entries of $A$.

\begin{lemma} Let $A$ be an $n \times n$ matrix with coefficients in $\mathbb{C}$. Then $|\det(A)| \leq  n! M(A)^n.$
\label{determinantbound}
\end{lemma}
\begin{proof} We have $$\det(A) = \sum_{\sigma \in S_n} sgn(\sigma) \prod_{i=1}^n A_{i,\sigma(i)}.$$ So
$$|\det(A)| \leq \sum_{\sigma \in S_n} |\prod_{i=1}^n A_{i,\sigma_i}| \leq \sum_{\sigma \in S_n} M(A)^n = n! M(A)^n.$$
\end{proof}

\begin{lemma}
\label{maximummatrixentry} Let $A$ and $B$ be $n \times n$ matrices with coefficients in $\mathbb{C}$. Then $M(AB) \leq nM(A)M(B)$. Thus, for any $m$, $M(A^m) \leq (nM(A))^m$.
\end{lemma}
\begin{proof}For any $i,j,k,l,$ $|A_{ij}B_{kl}| \leq M(A)M(B)$.  Since every entry in $AB$ arises from a sum of $n$ terms each of which is a product of an entry in $A$ and an entry $B$, the desired bound follows. The second part of the lemma follows from the first.

\end{proof}

Henceforth, we will assume that $O_K$ has a unit $a$ of infinite order (that is $K$ is not the rationals or a quadratic imaginary field).

\begin{lemma}For any $a$ in $O_K^{\times}$,  there exists a constant $C>1$ such that for all $k$, $$|N_{K/\mathbb{Q}}(a^k-1)| \leq C^k.$$
\label{fieldnormbound}
\end{lemma}
\begin{proof} Pick some Archimedean embedding of $K$ into $\mathbb{C}$ and let $\lambda_1, \lambda_2, \cdots \lambda_d$ be a basis for $K$ over $\mathbb{Q}$. Then for some $a_{ij}$ we have $$a\lambda_i=\sum_j a_{ij}\lambda_j.$$
 We have $N_{K/\mathbb{Q}}(a)= \det(A)$ where $A$ is the matrix with entries $a_{ij}$. Then, $$N_{K/\mathbb{Q}}(a^k)= \det(A^k).$$ By Lemma \ref{maximummatrixentry} $$M(A^k) \leq (dM(a))^k.$$ Let $B$ be the matrix with coefficients $b_{ij}$ defined by $$(a^k-1)\lambda_i =\sum_j b_{ij}\lambda_j.$$ So $N_{K/\mathbb{Q}}(a^k-1) = \det B$. Note that $B$ depends on $a$ and $k$.  By Lemma \ref{determinantbound}, we  have $$|N_{K/\mathbb{Q}}(a^k-1)| = |\det B| \leq d!M(B)^d.$$ Since $M(B) \leq M(A^k)+1$, we can apply Lemma \ref{maximummatrixentry} and Lemma \ref {determinantbound} so $$ |\det B| \leq d!(dM(A)^k+1)^d \leq (d!(1+dM(A)))^{dk}.$$ So we may take $C = (d!(1+dM(A)))^d$.

\end{proof}

 For any ideal $I$ of $O_K$ let $o_a(I)$ be the smallest positive integer such that $I|(a^{o_a}-1)$. So, $o_a$ is the order of $a$ in the group $(O_K/I)^{\times}$.

\begin{lemma}
 There is a constant $C> 1$ depending on $K$ and $a$ such that for any ideal $I$ of $O_K$,  we have $o_a \geq \log_{C} \textbf{N}I$.
 \label{idealcountlemma}
\end{lemma}
\begin{proof} Assume that $I|(a^{k}-1)$. So $\textbf{N}I \leq \textbf{N}(a^k-1) = |N_{K/\mathbb{Q}}(a^k-1)|$ since $(a^k-1)$ is a principal ideal. By Lemma \ref{fieldnormbound}, there is a constant $C>1$ depending only on $a$ such that $|N_{K/\mathbb{Q}}(a^k-1)| \leq C^k$ from which the result follows.

\end{proof}

\begin{corollary} Assume that the group of units of $O_K$ has positive rank. Then there is a constant $c>1$ such that
$[U_K:U_K(I)] \geq \log_{c} \textbf{N}I$.

\end{corollary}
\begin{proof} Apply Lemma \ref{idealcountlemma}, and note that if $u$ is a unit of infinite order, then $I|(u^{[U_K:U_K(I)]}-1)$.
\end{proof}

\begin{lemma} There exists a constant $k$ such that for all but finitely many ideals $I$, we have
$\omega(I) < k\frac{\log \textbf{N}I}{\log \log \textbf{N}I}$. Moreover, if we let $\tau(I)$ be the number of distinct ideal divisors of $I$ then
$\tau(I) = O((\textbf{N}I)^\epsilon)$ for any $\epsilon >0$.
\label{omegabound}
\end{lemma}
\begin{proof} Recall that for any $\epsilon>0$ and for all but finitely many $n$ we have $\omega(n) < \frac{(1+\epsilon)\log n}{\log \log n}$. As before, we will take $\epsilon=1$, and so for all but finitely many $n$, $\omega(n) < \frac{2\log n}{\log \log n}$. Let $N$ be the set of positive integers violating the prior inequality, and let $M$ be the set of ideals with norm in $N$ or with norm less than or equal to $e^e$. Note that since $N$ is finite, so is $M$. Let $C_M$ be the maximum number of distinct prime divisors of any element of $M$.

Now, for any given ideal $I$, set $I_0=rad(I)$ to the largest squarefree ideal divisor of $I$.

So we need to just estimate $\omega(I_0)$  If  $I_0 \in M$, then $\omega(I_0) \leq C_M$. If $I_0 \notin M$, then since any prime in $\textbf{Z}$ factors into at most $d$ distinct prime ideals in $K$, we have
$$\omega(I_0) \leq  \frac{2d\log \textbf{N}I_0}{\log \log \textbf{N}I_0} \leq \frac{2d \log  \textbf{N}I}{\log \log \textbf{N}I}.$$ The last inequality follows from the fact that for $x > e^e$, $\frac{\log x}{\log \log x}$ is an increasing function.

To prove the result for $\tau(I)$, note that $\tau(n)=O(n^\epsilon)$ and that it suffices to prove that there is a constant $m$, such that  $$\tau(I)= o(\textbf{N}(I)),$$ which we need to only prove for ideals whose norm is a power a prime since $\tau(I)$ is a multiplicative function.  Assume that $\textbf{N}(I)=p^a$, so that $\tau(\textbf{N}(I))=a+1$. Then the number of divisors of $I$ is bounded by the number of solutions in nonnegative integers to the inequality $$x_1 + x_2 + \cdots + x_d  \leq a+1$$ which is  $(a+d+1)!/(d!(a+1)!)$. Since $d=[K:Q]$ is fixed the result follows.
\end{proof}

\begin{theorem} For any $\alpha <3$ we have $P_K(x) = O(\frac{x^2}{(\log x)^\alpha})$.
\end{theorem}
\begin{proof} The proof is similar to our earlier estimate for $G(x)$.

We will show then the equivalent result that for any fixed $\alpha$ such that $1< \alpha < 3$, we have $P_K(x) =O(\frac{x^2 (\log \log x)^2}{(\log x)^\alpha})$.

 We will write $\mathrm{ord}(I)$ to be the smallest positive integer $o$ such that $a^o-1$ is in $I$. We note that $$P_K(x) \leq \sum_{\textbf{N}I \leq x} \frac{\textbf{N}I}{\mathrm{ord}(I)}$$ Let $c$ be the constant from Lemma \ref{idealcountlemma}. And let $\ell(x)=3\log_C x$, and set $t=(\log x)^\alpha$.

 We set $T(x,p,q)$ to be those $I$ with norm at most $x$ and satisfying $p \leq \mathrm{ord}(I) < q$

 $$P_K(x) \leq \mathrm{I}(x) + \mathrm{II}(x) + \mathrm{III}(x) $$ where we have sums defined in an analogous fashion as earlier. That is,
 \begin{equation}\label{Idef2} \mathrm{I}(x)= \sum_{I \in T(x,1,\ell(x))} \frac{{\bf{N}}I}{{\mathrm{ord}}(I)},\end{equation} \begin{equation}\label{IIdef2} \mathrm{II}(x)= \sum_{I \in T(x,\ell(x),t)} \frac{{\bf{N}}I}{{\mathrm{ord}}(I)}\end{equation} and \begin{equation}\label{IIIdef2} \mathrm{III}(x)= \sum_{I \in T(x,t, \infty) } \frac{{\bf{N}}I}{{\mathrm{ord}}(I)}.\end{equation}

We will first estimate $\mathrm{I}(x)$. The number of terms in $\mathrm{I}(x)$ is bounded by $\sum_{1 \leq r \leq \ell(x)} \tau((a^r-1))$. This sum is at most
$\ell(x)\tau_m(x)$ where $\tau_m(x)$ is the maximum of $\tau((a^r-1))$ over $r$ satisfying $1 \leq r \leq \ell(x)$. By Lemma \ref{omegabound} we have $\tau_m(I) = O(x^{1/2})$. Any term in $\mathrm{I}(x)$ is bounded above by $x$ and so conclude that $$I(x)= O(x^{1/2}3(\log_c x) x)= O(x^{3/2}\log x).$$

Next, we need to estimate $\mathrm{III}(x)$. $\mathrm{III}(x)$ has at most $O(x)$ terms (since the number of ideals of norm at most $x$ is $O(x)$) and each term is at most $\frac{x}{t}$ and so $$\mathrm{III}(x) =O\left(\frac{x^2}{t}\right)=O\left(\frac{x^2}{(\log x)^\alpha}\right).$$

To estimate $\mathrm{II}(x)$ we will break it into two sets of terms depending on how many distinct prime factors $I$ has, along with an $O(1)$ as before.

Let $k$ be a constant that satisfies Lemma \ref{omegabound}. We define $N_k$ to be the set of ideals which divide some $J$ with $\omega(J) \geq \frac{C\log \textbf{N}j}{\log \log \textbf{N}}$. Note that  $N_K$ is finite.

As before, fix $\beta=  \frac{\alpha -1}{2}$. Set $H(x)$ to be the subset
of elements of $T(x,\ell(x),t)$ with are not in $N_K$ and which have at most $(\log x)^\beta$ distinct prime factors and set  \begin{equation}\label{II1def2} \mathrm{II}_1(x) = \sum_{I \in H(x) } \frac{{\bf{N}}I}{\mathrm{ord}(I)}. \end{equation} Similarly,
let $J(x)$ to be the set of elements of $T(x,\ell(x),t)$ with more than $(\log x)^\beta$ distinct prime factors, and not in $N_K$ and set \begin{equation}\label{II2def2} \mathrm{II}_2(x) = \sum_{I \in J(x)} \frac{{\bf{N}}I}{\mathrm{ord}(I)}.\end{equation}

To estimate $\mathrm{II}_1(x)$ we will estimate $|H(x)|$. If $I \in H(x)$ then $I|((u^r-1))$ for some $r$ with $\ell(x) \leq r \leq t$
and $((a^r-1))$ not in $N_K$. Then by Lemma \ref{fieldnormbound} and Lemma \ref{omegabound}we have for some constant $M$, $$\omega((a^r-1)) < \frac{Mr}{\log r} < t.$$ The last inequality is valid for $t$ sufficiently large, which occurs when $x$ is sufficiently large.
However, each element of $|H(x)|$ has $j$ distinct prime divisors where $1 \leq j \leq (\log x)^\beta$. So for any fixed $r$, there are at most $\sum_{j=1}^m {k \choose j}$ possible choices for the distinct prime divisors where $k$ is the largest integer less than or equal to $t$
and $m$ is the largest integer less than or equal to $(\log x)^\beta$. Any prime can
be raised to at most the $\log_2 x$ power, so the total number of divisors is bounded by  $(\sum_{j=1}^m {k \choose j})(\log_2 x)^{(\log x)^\beta}$. There are at most $t$ possible values of $r$. Thus,
$$|H(x)| \leq \sum_{j=1}^m {k \choose j}t(\log_2 x)^{(\log x)^\beta}.$$   Applying Lemma \ref{HRlemma} and using the values of $k$ and $m$ then gives $$ |H(x)|\leq t\left(\frac{e^2 t }{(\log x)^\beta }\right)^{(\log x)^\beta} (\log_2 x)^{(\log x)^\beta}.$$

Now, since $t= (\log x)^\beta$, $$\log |H(x)| \leq C(\log x)^\beta(\log \log x) + \log t + (\log \log_2 x)(\log x)^\beta  $$ for some constant $C$ and this is $O((\log x)^\beta \log \log x)$. Thus $|H(x)| = O(x^{\epsilon_0})$, and so $\mathrm{II}_1(x) = O(x^{1+\epsilon_0})$ for any $\epsilon_0>0$.

To estimate  $\mathrm{II}_2(x)$, apply Lemma \ref{omegacountbound2} so that the number of terms of $J$ is $O(\frac{x (\log \log x)^2}{(\log x)^{2\beta}})$ and thus conclude that $$\mathrm{I}_2(x) = O\left(\frac{x^2 (\log \log x)^2}{(\log x)^{2\beta +1}}\right)$$ since every
term is at most $O(\frac{x}{\log_{c} x})$.

\end{proof}


\begin{thebibliography}{1}
\bibitem{Ambrose} Christopher Ambrose, On a problem of Rohrlich, to appear.

\bibitem{Hardy}
  G. H. Hardy and E. M. Wright, {\it An Introduction to the Theory
  of Numbers}, Oxford University Press, 1985.

\bibitem{Hooley}
Christopher Hooley, On Artin's conjecture, {\it J. Reine Angew. Math.} {\bf 225} (1967) 209–220.

\bibitem{Heath-Brown} D.R. Heath Brown, Artin's conjecture for primitive roots, {\it Quart. J. Math. Oxford Ser.} (2) {\bf 37} (1986), no. 145, 27–38.

\bibitem{Lang} Serge Lang, {\it Undergraduate analysis } (2nd edition), Springer-Verlag 1997.

\bibitem{Lang2} Serge Lang, {\it Algebraic Number Theory} (2nd edition), Springer-Verlag, 1994

\bibitem{Murty} Rajiv Gupta, M. Ram Murty, A remark on Artin's conjecture. {\it Invent. Math.} {\bf} 78 (1984), 1, 127–130.

\bibitem{Rohrlich} D.E Rohrlich, Self-Dual Artin Representations, {\it Automorphic Representations and L-Functions}, editors D. Prasad, C.S. Rajan, A. Sankaranarayanan, J. Sengupta, Hindustan Book Agency, (2013), to appear.


\end{thebibliography}
\end{document}